\documentclass[12pt]{amsart}
\usepackage{amsmath,amssymb,amsfonts}
\usepackage{fullpage}
\theoremstyle{plain}
\newtheorem{theorem}{Theorem}[section]

\newtheorem{proposition}{Proposition}[section]

\theoremstyle{definition}

\newtheorem{remark}{Remark}[section]
\newtheorem{example}{Example}[section]

\catcode`@=11 \@addtoreset{equation}{section} \catcode`@=12
\font\Bbb=msbm10 at 12pt

\newcommand{\R}{\mbox{\Bbb R}}
\newcommand{\N}{\mbox{\Bbb N}}

\begin{document}
\title[Infinite systems with nonlocal conditions]{Infinite first order differential systems with nonlocal initial conditions}


\subjclass[2010]{Primary 34A12, secondary 34A34, 47H30}%
\keywords{Infinite systems, Nonlocal initial condition, Fr\'{e}chet space, Fixed point.}%

\author[G. Infante]{Gennaro Infante}
\address{Gennaro Infante, Dipartimento di Matematica e Informatica, Universit\`{a} della
Calabria, 87036 Arcavacata di Rende, Cosenza, Italy}%
\email{gennaro.infante@unical.it}%

\author[P. Jebelean]{Petru Jebelean}%
\address{Petru Jebelean, Department of Mathematics, West University of
 Timi\c{s}oara, Blvd. V. P\^{a}rvan, no 4, 300223-Timi\c{s}oara,
 Rom{a}nia}%
\email{jebelean@math.uvt.ro}%

\author[F. Madjidi]{Fadila Madjidi}
\address{Fadila Madjidi, Universit\'{e} Abderrahmane Mira,
  Facult\'{e} des Sciences Exactes, Route de Targa-ouzemour, 06000 B\'{e}jaia, Algeria}%
\email{madjidifadila@gmail.com}%

\begin{abstract}
We discuss the solvability of an infinite system of first order ordinary differential equations on the half line, subject to nonlocal initial conditions. The main result states that if the nonlinearities possess a suitable ``sub-linear'' growth then the system has at least one solution. The approach relies on the application, in a suitable Fr\'{e}chet space, of the classical Schauder-Tychonoff fixed point theorem. We show that, as a special case, our approach covers the case of a system of a finite number of differential equations. An illustrative example of application is also provided.
\end{abstract}

\maketitle

\section{Introduction}
In the recent paper~\cite{[BIP]} Bolojan-Nica and co-authors developed a technique that can be used to study the solvability of a system of $N$ first order differential equations of the form
\begin{equation}
\left\{
\begin{array}
[c]{l}%
x_{1}^{\prime}\left(  t\right)  =g_{1}\left(  t,x_{1}\left(  t\right)
,x_{2}(t),...,x_{N}\left(  t\right)  \right),\\
...\\
x_{N}^{\prime}\left(  t\right)  =g_{N}\left(  t,x_{1}\left(  t\right)
,x_{2}\left(  t\right)  ,...,x_{N}\left(  t\right)  \right),
\end{array}
\right.  \quad t \in [0,1],  \label{1}%
\end{equation}
subject to the coupled  \emph{nonlocal} conditions%
\begin{equation}
\left\{
\begin{array}
[c]{l}%
x_{1}\left(  0\right)  =\displaystyle\sum_{j=1}^{N}\langle \eta_{1j},x_j \rangle,\\
...\\
x_{N}\left(  0\right)  =\displaystyle\sum_{j=1}^{N}\langle \eta_{Nj},x_j \rangle,
\end{array}
\right.  \label{1'}%
\end{equation}
where
$\eta_{ij}%
:C[0,1]\rightarrow\mathbb{R}$ are continuous linear
functionals. These functionals involve a suitable support and can be written in a form involving Stieltjes integrals, namely,
\begin{equation}\label{stjcon}
\langle \eta_{ij},v \rangle  =\int_{0}^{\hat{t}}v\left(  s\right)
dA_{ij}\left(  s\right), \quad \left(  v\in C\left[  0,1\right]  \right),
\end{equation}
where $\hat{t}\in[0,1]$ is given.
The formulation~\eqref{stjcon} covers, as special cases,
initial conditions of the type
\begin{equation}
\label{pc}\langle \eta_{ij},v \rangle
=\mathop{\displaystyle \sum }\limits_{k=1}^{m}\eta_{ijk}v\left(  t_{k}\right) ,
\end{equation}
where $\eta_{ijk}\in\mathbb{R}$ and $0\leq t_{1}<t_{2}<...<t_{m}\leq \hat{t}$. Note that, if all the functionals $\eta_{ij}$ have discrete expressions,
\eqref{1'} gives a multi-point condition. The approach in~\cite{[BIP]} relies on the fixed point theorems of Perov, Schauder and Schaefer
and on a vector method for treating systems which uses matrices with spectral radius less than one.

We mention that there exists a wide literature on differential equations subject to nonlocal conditions; we refer here to the pioneering work of
Picone~\cite{Picone}, the  reviews by Whyburn~\cite{Whyburn}, Conti~\cite{Conti}, Ma~\cite{rma},  Ntouyas~\cite{sotiris} and \v{S}tikonas~\cite{Stik}, the papers by Karakostas and Tsamatos~\cite{kttmna, ktejde} and by Webb and Infante~\cite{jwgi-lms,jwgi-lms-II}.

Note that on non-compact intervals the problem of finding solutions of differential systems becomes more complicated, since the associated integral formulation may display a lack of compactness. For systems of a finite number of differential equations on non-compact intervals, subject to linear (or more general) conditions, we refer to the papers by Andres et al.~\cite{[AGG]}, Cecchi and co-authors~\cite{[CFM], [CMZ]},
De Pascale et al.~\cite{[DLM]}, Marino and Pietramala~\cite{[MP]} and Marino and Volpe~\cite{[MV]}.

On the other hand, infinite systems of differential equations often occur in applications, for example
in stochastic processes and quantum mechanics \cite {[AB]} (also, see ch. XXIII in \cite{[HP]}) and in physical chemistry of macromolecules \cite{[O]}. For some recent work on infinite systems on compact intervals we refer the reader to the papers by Frigon~\cite{[F]}, by Bana{\'s} and co-authors~\cite{[B],[BL],[BO]} and references therein.

The methodology to treat initial value problems for infinite systems, often relies on using the theory of differential equations in Banach spaces, but it is also known that this approach can reduce the set of solutions and some of their specific properties. A fruitful alternative is to deal with such systems in the framework of locally convex spaces.
For a discussion emphasizing these aspects we refer the reader to the paper by Jebelean and Reghi\c{s}~\cite{[JR]} and references therein.

The case of initial value problems for infinite systems on non-compact intervals has been investigated in~\cite{[JR]}, where the authors extended a comparison theorem due to Stokes~\cite{S}, valid for the scalar case. Here we utilize the framework developed in~\cite{[JR]}, in order to deal with infinite systems of nonlocal initial value problems on the half line. To be more precise,
 let us denote by $J$ the  half line $[0,+\infty)$ and by $\mathcal{S}$ the space of all real valued sequences,  that is $$\mathcal{S}:=\mathbb{R}^{\mathbb{N}}=\{x=(x_1,x_2,...,x_n,...):\;x_n\in \mathbb{R},\;\forall\;n\in\mathbb{N}\},$$
which is assumed to be endowed with the product topology. We fix $t_0\in (0,+\infty)$ and consider the space $C[0,t_0]$ of all continuous real valued functions defined on $[0,t_0]$ with the usual supremum norm. For each  $n\in \mathbb{N},$ let $f_n:J\times \mathcal{S}\to \mathbb{R}$ be a continuous function and  $\alpha_n: C[0,t_0]\to \mathbb{R}$ be a continuous linear functional.

Here we deal with the solvability of the nonlocal initial value problem
\begin{equation}\label{IVP initial}\left\{
\begin{array}{l}
  x'_n(t) =  f_n(t,x_1(t), x_2(t), ..., x_n(t), ...) \; , \quad t\in J , \\
  \cr
  x_n(0)  =  \langle \alpha_n,x_n|_{[0,t_0]} \rangle,
\end{array}
\right. \qquad(n\in \mathbb{N})
\end{equation}
where, in a similar way as above, $\langle \alpha_n,v \rangle$ denotes the value of the functional $\alpha_n$ at $v\in C[0,t_0]$.
By a {\it solution} of problem \eqref{IVP initial} we mean a sequence of functions $x_n:J \to \R$, $n\in \N$, such that each $x_n$ is derivable on $J$ and satisfies  $$x'_n(t) =  f_n(t,x_1(t), x_2(t), ..., x_n(t), ...), \quad \forall t\in J,$$  together with the nonlocal initial condition $x_n(0)  =  \langle \alpha_n,x_n|_{[0,t_0]} \rangle$.

The rest of the paper is organized as follows. In Section 2 we transform the system \eqref{IVP initial} into a fixed point problem within a suitable Fr\'{e}chet space. In Section 3 we present the main existence result, we show how this approach can be utilized in the special case of a finite number of differential equations and, finally, we provide an example that illustrates our theory. Our results are new in the context of nonlocal problems for infinite systems and complement earlier results in the literature.

\section{An equivalent fixed point problem}
We firstly introduce some notations that are used in the sequel. For $x,y \in S,$ the product element $x \cdot y \in \mathcal{S}$ is defined by $x \cdot y=(x_1 \, y_1,x_2 \, y_2,...,x_n \, y_n,...)$; if $x$ is such that $x_n \neq 0$ for all $n\in \mathbb{N}$, then we denote by $x^{-1}=(x_1^{-1}, x_2^{-1}, ..., x_n^{-1},...)$.
We also set
$$
[x]_{n}=\max\limits_{1\leq i\leq n}|x_i|\; ,\quad \forall \; n\in \mathbb{N},
$$
and notice that the family of seminorms $([\; \cdot \; ]_n)_{n\in \mathbb{N}}$ on the space $\mathcal{S}$ generates a product topology. Endowed with this topology, $\mathcal{S}$ is a Fr\'{e}chet space, that is a locally convex space which is complete with respect to a translation invariant metric.

For the sake of simplicity, we rewrite \eqref{IVP initial} in a more compact form. In order to do this, we define the map $f:J\times \mathcal{S}\to \mathcal{S}$ by
$$
f(t,y)=(f_1(t,y),f_2(t,y),...,f_n(t,y),...) \; , \; \quad \forall\;(t,y)\in J\times \mathcal{S}
$$
and the operator $\alpha: C([0,t_0], \mathcal{S})=\prod\limits_{n=1}^{\infty}C[0,t_0]\to  \mathcal{S}$ by setting
\begin{equation}\label{alpha}
\langle \alpha ,v \rangle =( \langle \alpha_1v_1 \rangle , \langle \alpha_2,v_2\rangle,...,\langle \alpha_n,v_n\rangle,...)  , \!  \quad \forall\; v=( v_1, v_2,..., v_n,...)\in C([0,t_0], \mathcal{S}).
\end{equation}
Notice that $f$ is continuous and $\alpha$ is linear and continuous. Therefore the problem (\ref{IVP initial}) can be re-written as
\begin{equation}\label{IVP finale}\left\{
\begin{array}{l}
  x'(t) = f(t,x(t)) \; , \quad t\in J \\
  \cr
  x(0) = \langle \alpha ,x|_{[0,t_0]} \rangle
\end{array}
    \right.
\end{equation}
the derivation $'$ being understood component-wise.
Integrating \eqref{IVP finale} we get
\begin{equation*}
x(t)=c+\int\limits_{0}^{t}f(s,x(s))\,ds,
\end{equation*}
with some $c \in \mathcal{S}$. The nonlocal initial condition yields
\begin{equation*}
    c=c \cdot \left \langle \alpha ,\mathbb{I} \right \rangle + \Bigl \langle \alpha , \Bigl ( \int\limits_{0}^{ \cdot }f(s,x(s))\,ds \Bigr ) \Big |_{[0,t_0]} \Bigr \rangle,
\end{equation*}
where $\mathbb{I}=(1,1,...,1,...)$. Thus, assuming that
\begin{equation}\label{Hyp}
\langle \alpha_n,1 \rangle \neq 1 \; \quad \forall \; n\in \mathbb{N},
\end{equation}
we obtain
\begin{equation}\label{sol new}
    x(t)=(\mathbb{I}- \left \langle \alpha ,\mathbb{I} \right  \rangle )^{-1} \cdot  \Bigl \langle \alpha , \Bigl ( \int\limits_{0}^{ \cdot }f(s,x(s))\,ds \Bigr ) \Big |_{[0,t_0]} \Bigr \rangle+\int\limits_{0}^{t}f(s,x(s))\,ds.
 \end{equation}

We consider the space $C(J)$  of all continuous real valued functions defined on $J$ with the topology of the uniform convergence on compacta. Given $(t_k)_{k \in \mathbb{N}}$ a strictly increasing sequence in $(0, +\infty)$ with $t_k \to +\infty$ as $k \to \infty$, this topology is generated by the (countable) family of seminorms
 $$\nu_k(h)= \max \limits_{t\in [0,t_k]}|h(t)|, \qquad h\in C(J), \; k\in \mathbb{N}.$$
Note that $C(J)$ is a Fr\'{e}chet space.
 We denote by $C(J, \mathcal{S})$ the space of all $\mathcal{S}$--valued continuous functions on $J$. Since
 $$C(J, \mathcal{S})=\prod\limits_{n=1}^{\infty}C(J),$$
 the space $C(J, \mathcal{S})$ equipped with the product topology also becomes a Fr\'{e}chet space.

\begin{proposition}\label{fp} Under the assumption \eqref{Hyp} we have that:
\begin{enumerate}
\item[$(i)$] $x\in C(J, \mathcal{S})$ is a solution of problem \eqref{IVP initial} iff it is a fixed point of the operator
 $T:C(J, \mathcal{S})\to C(J, \mathcal{S})$ defined by
 \begin{equation}\label{oper}
    T(v)(t)=(\mathbb{I}-\langle \alpha, \mathbb{I}\rangle)^{-1}\cdot  \Bigl \langle \alpha , \Bigl ( \int\limits_{0}^{ \cdot }f(s,v(s))\,ds \Bigr ) \Big |_{[0,t_0]} \Bigr \rangle+\int\limits_{0}^{t}f(s,v(s))\,ds \; ;
    \end{equation}
\item[$(ii)$]  $T$ is continuous and maps bounded sets into relatively compact sets.
\end{enumerate}
\end{proposition}
\begin{proof} $(i)$ This fact follows by virtue of \eqref{sol new} and \eqref{oper}.\newline
$(ii)$ It is a standard matter to check that, for all $k\in \mathbb{N}$, the operator $R_k: C([0,t_k],\mathcal{S}) \to C([0,t_k],\mathcal{S})$ defined by
$$R_k(v)(t)=\int_0^tf(s,v(s))ds \; , \; v\in C([0,t_k],\mathcal{S}), \; t\in [0,t_k]$$
is continuous. We now show that $R:C(J,\mathcal{S}) \to C(J,\mathcal{S})$ given by
$$R(v)(t)=\int_0^tf(s,v(s))ds \; , \; v\in C(J,\mathcal{S}), \; t\in J$$
 inherits the same property. Indeed let $v \in  C(J,\mathcal{S})$ and $(v^m)_{m\in \mathbb{N}} \subset  C(J,\mathcal{S})$ be with $v^m \to v$, as $m\to \infty$. Then, one has that $v^m|_{[0,t_k]} \to v|_{[0,t_k]}$ in $C([0,t_k],\mathcal{S})$, for all $k\in \mathbb{N}$. Using the continuity of $R_k$, it follows
 $$R(v^m)|_{[0,t_k]} =R_k(v^m|_{[0,t_k]})\to R_k(v|_{[0,t_k]})=R(v)|_{[0,t_k]} \; ,\ \text{for all}\ k\in \mathbb{N},  $$
which means that $R(v^m) \to R(v)$ in $C(J,\mathcal{S})$.

We also have that $R$ maps bounded sets into relatively compact sets. To see this, let $M\subset C(J,\mathcal{S})$ be bounded. Using the Arzel\`a--Ascoli theorem, it is routine to show that for any fixed $n\in \mathbb{N}$, the set
$$\{R(v)_{n}|_{[0,t_k]}\,:\, v\in M \}=\{R_k(v|_{[0,t_k]})_{n}\,:\, v\in M \}$$  is relatively compact in $C[0,t_k]$ for each $k\in \mathbb{N.}$ Hence, the set $\{R(v)_{n}\,:\,v\in M \}$
is relatively compact in $C(J)$. By virtue of the Tychonoff theorem, we get that  $R(M)$ is relatively compact in the product space $C(J,\mathcal{S}).$

Now, the conclusion follows from the equality
$$T(v)= (\mathbb{I}-\langle \alpha,\mathbb{I}\rangle)^{-1}\cdot \left \langle \alpha , R(v)|_{[0,t_0]} \right \rangle+R(v)\; , \; \forall \; v\in C(J,\mathcal{S})$$
and the above properties of the operator $R$.
\end{proof}
We shall make use of the following consequence of the classical Schauder-Tychonoff fixed point theorem, see for example Corollary 2, page 107 of \cite{[D]}.
\begin{theorem}\label{ScTh}
Let $K$ be a closed convex set in a Hausdorff, complete  locally convex vector space and $F:K \mapsto K$ be continuous. If $F(K)$ is relatively compact then $F$ has a fixed point in $K$.
\end{theorem}
\section{Main result}
We present now the main result of the paper. In order to do this, we fix a strictly increasing sequence $(n_p)_{p\in \mathbb{N}} \subset \mathbb{N}$ and $(t_p) _{p\in \mathbb{N}}$ a sequence of real numbers, such that
\begin{equation}\label{t}
t_0<t_1<...<t_p \to +\infty \; , \; \mbox{ as }p\to \infty,
\end{equation}
and, for $\alpha$ as in \eqref{alpha}, we denote
\begin{equation}
\| \alpha\|_*= ( \| \alpha_1\|, \| \alpha_2\|,...,\| \alpha_n\|,...),
\end{equation}
where $\| \alpha_n\|$ stands for the norm of the functional $\alpha_n$ in the dual space of $C[0,t_0]$ ($n\in \mathbb{N}$).
\begin{theorem}\label{MR}
Assume that \eqref{Hyp} holds and that $f$ satisfies the condition
\begin{equation}\label{growth}
    [f(t,x)]_{n_{p}}\leq\left\{
\begin{array}{l}
  A_{p}(t)[x]_{n_{p}}+B_{p}\;,\ t\in [0,t_0], \\\\
  C_{p}([x]_{n_{p}}+1)\;,\ t\in [t_0,t_{p}],
\end{array}
\right. \qquad(x\in \mathcal{S} \; , \; p\in \mathbb{N}),
\end{equation}
with $A_{p}\in L^{1}_+(0,t_0)$  and $B_{p},\;C_{p}\in \mathbb{R}^{+}.$  If the inequality
\begin{equation}\label{less than1}
    \{[(\mathbb{I}- \langle \alpha , \mathbb{I} \rangle)^{-1}]_{n_{p}}[ \; \| \alpha \|_* ]_{n_{p}}+1 \} \|A_{p}\|_{L^{1}} < 1
\end{equation}
holds for every $p\in \mathbb{N}$, then the problem \eqref{IVP finale} (hence \eqref{IVP initial}) has at least one solution $x\in C(J, \mathcal{S}).$
\end{theorem}

\begin{proof} Taking into account Proposition~\ref{fp}, it suffices to show that $T$ has a fixed point in $C(J, \mathcal{S}).$  In order to do this, we shall apply Theorem~\ref{ScTh}.
Firstly, note that
\begin{equation}\label{sem}
\begin{array}{lll}
\displaystyle
 \Bigl [ \Bigl \langle \alpha , \Bigl ( \int\limits_{0}^{ \cdot }u(s)\,ds \Bigr ) \Big |_{[0,t_0]} \Bigr \rangle  \Bigr ]_n  & = & \displaystyle \max \limits_{1\leq i \leq n} \Bigl | \Bigl \langle \alpha_i , \Bigl ( \int\limits_{0}^{ \cdot }u_i(s)\,ds \Bigr ) \Big |_{[0,t_0]} \Bigr \rangle \Bigr |\\
& \leq &  \displaystyle\left [ \; \|\alpha \|_* \right ]_{n} \int \limits_0^{t_0} \left [ u(s) \right ]_nds \; ,
\end{array}
\end{equation}
for all $u\in C(J, \mathcal{S})$ and $n\in \mathbb{N}$.

For each  $p\in \mathbb{N},$  let
$$G_{p}:=[(\mathbb{I}- \langle \alpha , \mathbb{I} \rangle)^{-1}]_{n_{p}}[ \; \| \alpha \|_* ]_{n_{p}}+1$$
and
$\theta _p >0$ be such that
$$M_{{p}} := G_{p}\|A_{p}\|_{L^1}+\frac{C_{p}}{\theta_{p}}<1;$$
note that such a $\theta _p $ exists from \eqref{less than1}.

We introduce the (continuous) semi-norms
$$\mathcal{P}_p(x)=\max \limits_{0 \leq t \leq t_0}[x(t)]_{n_{p}} \; ,\quad \mathcal{Q}_p(x)=\max \limits_{t_0 \leq t \leq t_{{p}}}e^{-\theta_{{p}}(t-t_{0})} [x(t)]_{n_{p}},$$
$$\mathcal{R}_{p}(x)=\max \{ \mathcal{P}_p(x) , \; \mathcal{Q}_p(x) \}, \quad x\in C(J, \mathcal{S}) \; , \; p\in \mathbb{N}.$$
We denote by
$$K_{p}: = G_{p}t_0B_{p}+C_p(t_{{p}}-t_0),$$
we take $\rho_p$ such that
\begin{equation}\label{Rmn}
    \rho_{p}\geq (1-M_{{p}})^{-1}K_{p},
\end{equation}
and define
$$\Omega_{p}:=\{x\in C(J, \mathcal{S}) \,:\, \mathcal{R}_{p}(x)\leq \rho_{p}\} \; , \quad p\in \mathbb{N}.$$
Since each $\Omega_{p}$ is closed and convex,  the (nonempty) set
$$\Omega:=\bigcap\limits_{p=1}^{\infty}\Omega_{p}$$
has the same properties. Next, we prove that $T(\Omega )\subset \Omega$.

Let $x \in \Omega $ and $p\in \mathbb{N}$. If $t\in [0,t_0]$, using \eqref{sem}, we have
$$
\begin{array}{lll}
 \displaystyle  [T(x)(t)]_{n_{p}} & \leq &  \displaystyle \Bigl [(\mathbb{I}-\langle \alpha, \mathbb{I}\rangle)^{-1}\cdot  \Bigl \langle \alpha , \Bigl ( \int\limits_{0}^{ \cdot }f(s,x(s))\,ds \Bigr ) \Big |_{[0,t_0]} \Bigr \rangle \Bigr ]_{n_{p}}  +  \displaystyle \Bigl [\int\limits_{0}^{t}f(s,x(s))\,ds \Bigr ]_{n_{p}}\\
   & \leq & \displaystyle [(\mathbb{I}-\langle \alpha, \mathbb{I} \rangle)^{-1}]_{n_{p}}[ \; \|\alpha\|_* ]_{n_{p}}\int\limits_{0}^{t_0}[f(s,x(s))]_{n_{p}}\,ds+ \int\limits_{0}^{t_0} \left [f(s,x(s)) \right ]_{n_{p}}\,ds \\
   & = & \displaystyle G_p\int\limits_{0}^{t_0}[f(s,x(s))]_{n_{p}}\,ds \; .
\end{array}
$$
The growth condition (\ref{growth}) gives
$$
   [T(x)(t)]_{n_{p}} \leq  G_p \{ \|A_{p}\|_{L^{1}}\mathcal{P}_p(x)+t_0B_{p}\}
$$
which, taking the maximum  over $[0,t_0],$ yields
\begin{equation}\label{ct0}
    \mathcal{P}_p(T(x))\leq G_p \{ \|A_{p}\|_{L^{1}}\mathcal{P}_p(x)+t_0B_{p}\}.
\end{equation}

If $t\in [t_0,t_p]$, then arguing as above we get
\begin{equation}\label{inq}
  \begin{array}{lll}
 \displaystyle  [T(x)(t)]_{n_{p}} & \leq & \displaystyle G_p \{ \|A_{p}\|_{L^{1}}\mathcal{P}_p(x)+t_0B_{p}\}+\int\limits_{t_0}^{t} \left [f(s,x(s)) \right ]_{n_{p}}\,ds\\
 & \leq & \displaystyle    G_p \{ \|A_{p}\|_{L^{1}}\mathcal{P}_p(x)+t_0B_{p}\}+C_p\int\limits_{t_0}^{t}\bigl ( \left [x(s) \right ]_{n_{p}}+1 \bigr )\,ds \\
 & \leq & \displaystyle    G_p \{ \|A_{p}\|_{L^{1}}\mathcal{P}_p(x)+t_0B_{p}\} \\
 &  & + \displaystyle C_p \int\limits_{t_0}^{t} e^{\theta_{{p}}(s-t_{0})}e^{-\theta_{{p}}(s-t_{0})} \left [x(s) \right ]_{n_{p}}\, ds +C_p(t_{p}-t_0) \\
  & \leq & \displaystyle    G_p \{ \|A_{p}\|_{L^{1}}\mathcal{P}_p(x)+t_0B_{p}\} +C_p\Bigl \{ \mathcal{Q}_p(x)\frac{e^{\theta_{{p}}(t-t_{0})}}{\theta _p}+t_{p}-t_0\Bigr \}.
 \end{array}
\end{equation}
Multiplying the inequality~\eqref{inq} by $\displaystyle e^{-\theta_{{p}}(t-t_{0})} \; (\leq 1)$, we obtain
$$[T(x)(t)]_{n_{p}}e^{-\theta_{{p}}(t-t_{0})} \leq G_p \{ \|A_{p}\|_{L^{1}}\mathcal{P}_p(x)+t_0B_{p}\} +C_p\Bigl \{ \frac{\mathcal{Q}_p(x)}{\theta _p}+t_{p}-t_0 \Bigr \},$$ and, taking the maximum over $[t_0,t_p]$, we obtain
\begin{equation}\label{ct1}
\mathcal{Q}_p(T(x)) \leq G_p \{ \|A_{p}\|_{L^{1}}\mathcal{P}_p(x)+t_0B_{p}\} +C_p\Bigl \{ \frac{\mathcal{Q}_p(x)}{\theta _p}+t_{p}-t_0\Bigr \}.
\end{equation}
From \eqref{ct0}, \eqref{ct1} and \eqref{Rmn}, we obtain
$$ \begin{array}{lll} \mathcal{R}_p(T(x)) & \leq & \displaystyle G_p \{ \|A_{p}\|_{L^{1}}\mathcal{R}_p(x)+t_0B_{p}\} +C_p\Bigl \{ \frac{\mathcal{R}_p(x)}{\theta _p}+t_{p}-t_0\Bigr \} \\
& = & M_p \mathcal{R}_p(x) + K_p \\
& \leq & M_p \rho_p + \rho_p(1-M_p) = \rho _p,
\end{array}
$$
showing that $x\in \Omega _p$, for all $p\in \mathbb{N}$; hence, $x \in \Omega$. Finally, note that $\Omega$ is a bounded subset of $C(J, \mathcal{S})$.  Then, from Proposition~\ref{fp}, the continuous operator  $T$ maps $\Omega$ into a relatively compact set. Therefore, by virtue of Theorem~\ref{ScTh}, $T$ has a fixed point in $\Omega$ and the proof is complete.
\end{proof}

\begin{remark} Note that the result above is also valid in the case of a non-compact interval $J_b=[0,b)$ with $t_0\in (0,b)$, $b\in (0,+\infty)$, instead of $J$. The only modification is to consider $t_p \to b$ in \eqref{t}, instead of $t_p \to +\infty$, as $p\to \infty$.
\end{remark}

We  now illustrate how this methodology can be applied in the case of a finite system of differential equations. The assumptions are relatively easy to be checked and we include this result for completeness.

We fix $N\in \N$ and discuss the solvability of the nonlocal initial value problem
\begin{equation}\label{IVP initialf}\left\{
\begin{array}{l}
  x'_n(t) = g_n(t,x_1(t), x_2(t), ..., x_N(t)) \; , \quad t\in J ; \\
  \cr
  x_n(0) = \langle \eta_n,x_n|_{[0,t_0]} \rangle,
\end{array}
\right. \qquad(n = 1\ldots N)
\end{equation}
where  $g_n:J\times \R^N \to \mathbb{R}$ are continuous functions and  $\eta_n:C[0,t_0]\to \mathbb{R}$ are continuous linear functionals, satisfying
\begin{equation}\label{eta}
\langle \eta_n,1 \rangle \neq 1\; , \quad n = 1\ldots N.
\end{equation}

We denote by
$$\| y \|_{\infty }:=\max\limits_{1\leq n\leq N}|y_n|\; ,\quad \forall \; y=(y_1,...,y_N)\in \mathbb{R}^N$$
and by
$$\eta_*:=( \| \eta_1\|, ...,\| \eta_N\|) \; ,\quad \underline{\eta}:=\Bigl ( \frac{1}{1- \langle \eta _1,1 \rangle} , ..., \frac{1}{1- \langle \eta _N,1 \rangle}\Bigr ).$$
With the notation above we can state the following existence result.

\begin{theorem}\label{MR1}
Let  $(t_p) _{p\in \mathbb{N}}$ be a sequence satisfying \eqref{t}. Assume the condition~\eqref{eta} holds and that  $g:=(g_1,...,g_N)$ satisfies
\begin{equation}\label{growthf}
    \|g(t,x)\|_{\infty}\leq\left\{
\begin{array}{lll}
  A(t)\|x\|_{\infty}+B\;,\ t\in [0,t_0], \\\\
  C_{p}(\|x\|_{\infty}+1)\;,\ t\in [t_0,t_{p}],
\end{array}
\right. \qquad( x\in \R^N \; ,\; p\in \mathbb{N}),
\end{equation}
with $A\in L^{1}_+(0,t_0)$  and $B,\;C_{p}\in \mathbb{R}^{+}.$  If the inequality
\begin{equation}\label{less than1f}
    \{\| \underline{\eta} \|_{\infty}\| \eta_* \|_{\infty}+1 \} \|A\|_{L^{1}} < 1
\end{equation}
holds, then the nonlocal initial value problem~\eqref{IVP initialf} has at least one solution.
\end{theorem}

\begin{proof} Consider the problem~\eqref{IVP finale} with
$$f=(g_1,g_2,...,g_N,g_{N}, ...)\; , \quad \alpha= (\eta_1, \eta_2,...,\eta_N,\eta_N,...)$$
and note that if $x=(x_1,x_2,...,x_N,...)\in C(J, \mathcal{S})$ is a solution, then $(x_1,x_2,...,x_N)$ solves~\eqref{IVP initialf}. To prove the solvability of \eqref{IVP finale} with the above choices of $f$ and $\alpha$, we apply Theorem~\ref{t} with $n_p=N+p-1$, $\forall p\in \N$. In this view, it follows from \eqref{growthf} that the growth condition~\eqref{growth} is fulfilled with $A_p=A$ and $B_p=B$. Also, we have
$$\{[(\mathbb{I}- \langle \alpha , \mathbb{I} \rangle)^{-1}]_{n_{p}}[ \; \| \alpha \|_* ]_{n_{p}}+1 \} \|A_{p}\|_{L^{1}}=\{\| \underline{\eta} \|_{\infty}\| \eta_* \|_{\infty}+1 \} \|A\|_{L^{1}},$$
for all $p\in \mathbb{N}$. Hence, the condition~\eqref{less than1} is satisfied from~\eqref{less than1f}  and the proof is complete.
\end{proof}
We conclude the paper with an example, where notice that the constants that occur in Theorem~\ref{t} can be effectively computed.
\begin{example}\label{ex1}
Consider the  nonlocal initial value problem:
\begin{equation}\label{ex}\left\{
\begin{array}{l}
  x_n'(t)  =  \displaystyle \frac{k_{n}}{1+t^{2}}\,x_n+ t\cos x_{n+1}, \; \quad t\in J ; \\
  \cr
  x_n(0)  = \displaystyle  \frac{1}{n+t_0}\,\int\limits_{0}^{t_0}x_n(s)\,ds.
\end{array}
\right. \qquad(n\in \mathbb{N}),
\end{equation}
The system \eqref{ex}
has at least one solution, provided that $k =(k_1,k_2,...,k_n,...)\in \mathcal{S}$ satisfies
\begin{equation}\label{arc}
    |k_{n}| < \frac{1}{(1+t_{0})\arctan t_{0}} \, , \quad \forall n\in \N.
\end{equation}
To see this, we apply Theorem \ref{MR} with
  $$
f_n(t,x_1,x_2,...,x_n,...)=\displaystyle \frac{k_{n}}{1+t^{2}}\,x_n+t\cos x_{n+1},
$$
$$
\langle \alpha_n, v \rangle= \displaystyle \frac{1}{n+t_0}\,\int\limits_{0}^{t_0}v(s)\,ds \; , \quad \forall v\in C[0,t_0] ,
$$
$n_p=p$ and  $(t_p) _{p\in \mathbb{N}}$ an arbitrary sequence satisfying \eqref{t}.
We have
\begin{equation}\label{e1}
\| \alpha _n \|=\langle \alpha_n, 1 \rangle = \displaystyle \frac{t_0}{n+t_0}<1,\quad \forall\,\,n\in \mathbb{N},
\end{equation}
and therefore \eqref{Hyp} is satisfied. Also, the growth condition \eqref{growth} is fulfilled with
\begin{equation}\label{e2}
\displaystyle A_p(t)=\frac{[k]_p}{1+t^2}, \quad B_p=t_0 , \quad C_p=[k]_p+t_p , \quad \forall p\in \N.
\end{equation}
Using \eqref{e1} and \eqref{e2} we can compute
\begin{equation*}
    \{[(\mathbb{I}- \langle \alpha , \mathbb{I} \rangle)^{-1}]_{{p}}[ \; \| \alpha \|_* ]_{{p}}+1 \} \|A_{p}\|_{L^{1}}=[k]_p(1+t_0)\arctan t_{0}
\end{equation*}
and \eqref{less than1} is accomplished from \eqref{arc}.
\end{example}

\section*{Acknowledgments}
G. Infante was partially supported by G.N.A.M.P.A. - INdAM (Italy).

\end{document}